\documentclass[11pt,letterpaper]{amsart}

\usepackage{amssymb}
\usepackage{amscd}
\usepackage{amsmath,amsfonts}
\usepackage{enumerate}
\usepackage{tikz}
\usepackage{subfig}
\usepackage[all]{xy}
\usepackage{todonotes}
\usepackage{dsfont}
\usepackage{comment}
\usepackage[pagebackref,bookmarksopen]{hyperref}

\usepackage{nicematrix}

\hypersetup{
	colorlinks=true,
	linkcolor=blue,
	citecolor=magenta,
	urlcolor=cyan,
}
\newtheorem{theo}{Theorem}[section]
\newtheorem{lemma}[theo]{Lemma}
\newtheorem{cor}[theo]{Corollary}
\newtheorem{prop}[theo]{Proposition}

\theoremstyle{definition}
\newtheorem{defi}[theo]{Definition}

\newtheorem{example}[theo]{Example}

\newtheorem{remark}[theo]{Remark}

\newcommand{\B}{{\mathbb{B}}}
\newcommand{\C}{{\mathbb{C}}}
\newcommand{\R}{{\mathbb{R}}}

\newcommand{\D}{\mathbb{D}}
\newcommand{\PP}{\mathbb{P}_{a,b}}

\DeclareMathOperator{\Disc}{Disc}
\DeclareMathOperator{\tr}{tr} % opcional, para consistência

\begin{document}

\title[Perplex Analysis and Geometry of Singularities]{Perplex Analysis and Geometry of Singularities}

%    author two information
\author[A.~Menegon]{Aur\'elio Menegon}
\address{Mid Sweden University \\ Department of Engineering, Mathematics and Subject Didactics \\ 85170 Sundsvall, Sweden }
\curraddr{}
\email{aurelio.menegon@miun.se}
%\thanks{The second author had partial support from CNPq (Brazil).}

% \keywords{Milnor fibration, $d$-regularity, differentiable singularities.} % (opcional)
%    For articles to be published after 1 January 2010, you may use
%    the following version:
%\subjclass[2010]{Primary }

%\keywords{Milnor fibration, $d$-regularity, differentiable singularities.}

\date{\today}

\begin{abstract}
We develop a real–analytic framework, called \emph{perplex analysis}, in which the complex, split–complex, and dual numbers arise as members of a single four–parameter family of two–dimensional commutative real algebras. Within this unified setting we define differentiability through a \emph{generalized Cauchy–Riemann structure}, extending several features of complex geometry to a broader real–analytic context. Two main results illustrate the analytic and geometric scope of the theory: a \L{}ojasiewicz gradient inequality for perplex–analytic functions, providing quantitative control of critical behavior; and a Milnor–Lê type fibration theorem for nondegenerate algebras, describing the local topology of singularities. The framework reveals a continuous transition between complex and hyperbolic geometries, with the dual boundary exhibiting new infinitesimal phenomena linked to zero divisors. These results connect generalized complex geometry, hypercomplex analysis, and singularity theory within a single analytic formalism.
\end{abstract}

\maketitle

\section*{Introduction}

The theory of complex numbers provides one of the most powerful frameworks in mathematics. 
Beyond their algebraic structure, holomorphic functions possess an analytic rigidity that underlies deep results in geometry and topology. 
A classical example is Milnor’s fibration theorem, which describes the local topology of complex hypersurface singularities and forms the analytic foundation of much of modern singularity theory.

In contrast, the real analytic setting lacks a natural notion of holomorphicity capable of producing a comparable theory. 
While singularities of real analytic maps play a central role in geometry and applications, the analytic tools available are significantly less rigid than in the complex case. 
Bridging this gap—developing an analytic formalism on the real plane that retains enough structure to support a singularity theory—has long remained an open problem.

The purpose of this article is to introduce a real–analytic framework that advances in this direction. 
We construct a four–parameter family of two–dimensional commutative real algebras, called \emph{perplex algebras}, which unify within a single structure the three classical quadratic models: the complex numbers, the split–complex (hyperbolic) numbers, and the dual numbers. 
This unified algebraic setting provides a common ground where analytic notions can be defined and studied in parallel across all these regimes. 
Within it, we formulate a natural notion of differentiability governed by a generalized Cauchy–Riemann equation, extending the classical concept of holomorphicity to a broader real–analytic context. 
The resulting \emph{perplex–analytic functions} retain key features of holomorphic maps while allowing for a controlled degeneration toward the dual limit.

This analytic structure leads naturally to geometric applications. 
We prove two main results that illustrate the strength of the framework. 
The first, Theorem~\ref{theo_Loja}, establishes a \L{}ojasiewicz gradient inequality for perplex–analytic functions, providing quantitative control of critical behavior and ensuring the analytic regularity required for geometric arguments. 
The second, Theorem~\ref{theo_MilnorFibration}, is a Milnor–Lê type fibration theorem for nondegenerate perplex algebras, describing the local fibration structure that organizes the topology of singularities in this setting. 
Together, these results show that the combination of algebraic structure and analytic control suffices to recover much of the geometric mechanism that operates in complex singularity theory.

Our approach is closely related to \emph{Hypercomplex Function Theory} (HFT), which studies analysis over fixed real quadratic algebras such as the complex, hyperbolic, or dual numbers. 
While HFT treats each case separately, the present work integrates them into a single parameter family in which these models appear as special points. 
This continuous–family perspective allows analytic results to be proved uniformly across the field, hyperbolic, and degenerate regimes, thereby revealing how complex and split–complex analysis arise as different faces of one analytic geometry, and how new phenomena emerge in the degenerate (dual) boundary.

The paper is organized as follows. 
Section~2 introduces the algebraic structure of perplex algebras and analyzes their parameter space, units, and classification into field, hyperbolic, and degenerate types. 
Section~3 develops the differential calculus and establishes the generalized Cauchy–Riemann equations. 
Section~4 proves the \L{}ojasiewicz inequality for perplex–analytic functions. 
Section~5 establishes the Milnor–Lê fibration theorem. 
Finally, Section~6 discusses perspectives and open problems at the intersection of analysis, geometry, and singularity theory, emphasizing how the perplex framework unifies these themes within a single analytic formalism.

\section{Perplex algebras}\label{sec:perplex}

We introduce the algebraic framework of \emph{perplex numbers}, a family of two–dimensional commutative, associative, unital $\R$-algebras obtained from a bilinear product on $\R^2$ determined by two triples of parameters $a,b\in\R^3$. We characterize precisely when such a product yields a perplex algebra, describe its parameter space, and record basic structural features: units, norm, conjugation, and quantitative estimates. Throughout, vectors in $\R^2$ are columns and the standard basis is implicit.

\medskip
\noindent\textbf{Construction.}
Given $a=(a_1,a_2,a_3)$ and $b=(b_1,b_2,b_3)$ in $\R^3$, define a bilinear product $*$ on $\R^2$ by
\[
\begin{bmatrix} x_1 \\ x_2 \end{bmatrix}
*
\begin{bmatrix} y_1 \\ y_2 \end{bmatrix}
:=
\begin{bmatrix}
 a_1x_1y_1 + a_2(x_1y_2+x_2y_1) + a_3x_2y_2 \\
 b_1x_1y_1 + b_2(x_1y_2+x_2y_1) + b_3x_2y_2
\end{bmatrix}.
\]
This product is bilinear, commutative, and distributive over addition; scalar multiplication is the standard one in $\R^2$.

\begin{prop}\label{prop_perplex}
The commutative $\R$-algebra $(\R^2,+,*)$ is associative, unital, and has units on both coordinate axes if and only if
\[
\begin{aligned}
&\text{(i)}\quad a_1a_3-a_2^2 \neq 0,\qquad
&&\text{(ii)}\quad a_1b_2-a_2b_1 \neq 0,\\
&\text{(iii)}\quad a_2b_2-a_3b_1 = 0,\qquad
&&\text{(iv)}\quad a_1a_3-a_2^2+a_2b_3-a_3b_2 = 0 ;
\end{aligned}
\]
or $a_1=b_2 \neq 0$ and $a_2=b_1=0$.
\end{prop}

\begin{defi}
When \emph{(i)}–\emph{(iv)} hold, we call $(\R^2,+,*)$ a \emph{perplex algebra} and denote it by $\PP$. Its elements are the \emph{perplex numbers}.
\end{defi}

\medskip
\noindent\textbf{Parameter space.}
Let $\mathcal{P}\subset\R^3\times\R^3$ be the semialgebraic locus of pairs $(a,b)$ satisfying Proposition~\ref{prop_perplex}. The two equalities in \emph{(iii)}–\emph{(iv)} have Jacobian of constant rank $2$ along $\mathcal{P}$; thus $\mathcal{P}$ is a smooth $4$-dimensional submanifold of $\R^6$. 

\medskip
\noindent\textbf{Identity and units.}
The multiplicative identity is
\[
\mathds{1}
= \frac{1}{a_1b_2-a_2b_1}
\begin{bmatrix} b_2 \\[2pt] -\,b_1 \end{bmatrix}.
\]
Write $\PP^\times$ for the unit group. One has $(\alpha x)^{-1}=\alpha^{-1}x^{-1}$ and $(x*y)^{-1}=x^{-1}*y^{-1}$ for $\alpha\in\R^\times$ and $x,y\in\PP^\times$. A direct computation shows that $x\notin\PP^\times$ iff
\[
x_1^2(a_1b_2-a_2b_1)+x_1x_2(a_1b_3-a_3b_1)+x_2^2(a_2b_3-a_3b_2)=0,
\]
a real conic in $\R^2$.

\medskip
\noindent\textbf{Powers and nilpotents.}
For $x\in\PP$ and $n\in\mathbb{N}$, set $x^n:=x*\cdots*x$ ($n$ times). Associativity implies the binomial identity
\[
(x+y)^n=\sum_{k=0}^n \binom{n}{k}\,x^k*y^{\,n-k}.
\]

\begin{prop}\label{prop:no-nilpotent}
$\PP$ has no nonzero nilpotent elements if and only if
\[
\Delta \;=\; (a_1 b_3 - a_3 b_1)^2 - 4 (a_1 b_2 - a_2 b_1)(a_2 b_3 - a_3 b_2) \;\neq\; 0.
\]
\end{prop}

\begin{proof}
Consider the left-multiplication operator $L_x(y):=x*y$. If $x$ were nilpotent, then $L_x$ would be nilpotent on a $2$-dimensional space, hence $L_x^2=0$, which forces $x^2=0$. Writing $x=(x_1,x_2)$, this gives two quadratic equations in $(x_1,x_2)$. 
Such a system has a nontrivial solution if and only if the quadratics 
\[
q_a(t)=a_1+2a_2 t+a_3 t^2, \qquad q_b(t)=b_1+2b_2 t+b_3 t^2
\]
have a common root. 
This occurs exactly when their resultant vanishes, i.e.\ when $\Delta=0$. 
\end{proof}

\begin{example}
Consider $a=(1,0,-1)$ and $b=(0,1,2)$. This pair satisfies (i)--(iv). Yet, in the corresponding algebra $\PP$ we compute
\[
(1,-1)^2=(0,0).
\]
Thus $(1,-1)$ is a nonzero nilpotent element. Indeed, in this case the discriminant equals $\Delta=0$, so we are exactly on the boundary between the ``field'' and the ``non-field'' regions. 
\end{example}

\medskip
\noindent\textbf{Classification of perplex algebras}

The previous discussion shows that the discriminant
\[
\Delta \;=\; (a_1 b_3 - a_3 b_1)^2 \;-\; 4 (a_1 b_2 - a_2 b_1)(a_2 b_3 - a_3 b_2)
\]
controls the basic structure of a perplex algebra $\PP$.  
We now prove that every perplex algebra is isomorphic to one of the three classical
two--dimensional real algebras: the complex numbers $\C$, the hyperbolic numbers
$\R \oplus \R$, or the dual numbers $\R[\varepsilon]/(\varepsilon^2)$.

\begin{prop}[Classification]\label{prop:classification}
Let $\PP$ be a perplex algebra. Then:
\begin{enumerate}[(i)]
\item If $\Delta < 0$, then $\PP \cong \C$. In this case we call $\PP$ a \emph{field-perplex algebra}.
\item If $\Delta > 0$, then $\PP \cong \R \oplus \R$. In this case we call $\PP$ a \emph{hyperbolic-perplex algebra}.
\item If $\Delta = 0$, then $\PP \cong \R[\varepsilon]/(\varepsilon^2)$. In this case we call $\PP$ a \emph{degenerate-perplex algebra}.
\end{enumerate}
\end{prop}

\begin{proof}
Consider the quadratic forms
\[
q_a(t) = a_1 + 2a_2 t + a_3 t^2,
\qquad
q_b(t) = b_1 + 2b_2 t + b_3 t^2.
\]
As noted in Proposition~\ref{prop:no-nilpotent}, the common roots of $q_a$ and $q_b$
determine the nilpotent directions in $\PP$, and the discriminant
\[
\Delta \;=\; (a_1 b_3 - a_3 b_1)^2 - 4 (a_1 b_2 - a_2 b_1)(a_2 b_3 - a_3 b_2)
\]
is precisely the resultant of these two quadratics.

Let $e_1=\binom{1}{0}$ and $e_2=\binom{0}{1}$ denote the standard basis of $\R^2$, viewed as elements of $\PP$. 
By definition of the multiplication, the operators of left multiplication satisfy
\[
L_{e_1}=A:=\begin{pmatrix}a_1 & a_2 \\ b_1 & b_2\end{pmatrix},
\qquad
L_{e_2}=B:=\begin{pmatrix}a_2 & a_3 \\ b_2 & b_3\end{pmatrix}.
\]
Condition (ii) of Proposition~\ref{prop_perplex} gives $\det A=a_1b_2-a_2b_1\neq 0$, 
hence $e_1$ is a unit. 
Thus there exists $e_1^{-1}\in\PP$ with $L_{e_1^{-1}}=A^{-1}$. 
Define
\[
j:=e_2*e_1^{-1}.
\]
Then
\[
L_j = L_{e_2}\circ L_{e_1^{-1}} = B\,A^{-1}.
\]

The assignment $x\mapsto L_x$ is injective (since $L_x(\mathds{1})=x$), 
so $j$ is uniquely determined by this property. 
The minimal polynomial of $j$ is
\[
\chi_j(\lambda)=\lambda^2-\tr(L_j)\,\lambda+\det(L_j),
\]
and a direct calculation shows that its discriminant is
\[
\Disc(\chi_j) := (\tr L_j)^2 - 4 \det L_j \;=\; (\det A)^{-2}\,\Delta.
\]

\smallskip
\noindent
-- If $\Delta < 0$, then $\chi_j$ is irreducible over $\R$, so the subalgebra
$\R[j]\subset \PP$ generated by $j$ is a $2$--dimensional division algebra over $\R$.
By uniqueness of the quadratic extension of $\R$, this identifies $\PP$ with $\C$.

\smallskip
\noindent
-- If $\Delta > 0$, then $\chi_j$ has two distinct real roots. 
The corresponding eigenvectors yield two orthogonal idempotents, 
producing a splitting $\PP \cong \R\oplus \R$, i.e.\ the hyperbolic numbers.  

\smallskip
\noindent
-- If $\Delta = 0$, then $\chi_j$ has a repeated real root, so $j$ has a nontrivial Jordan block. 
This produces a nilpotent element and identifies $\PP$ with the dual numbers.  

\smallskip
Thus every perplex algebra is isomorphic to exactly one of the three classical quadratic models: 
the complex, hyperbolic, or dual numbers.
\end{proof}

\begin{defi}
A perplex algebra is said to be \emph{nondegenerate} if it is either field-perplex or hyperbolic-perplex.
\end{defi}

\medskip
\noindent\textbf{Perplex norm.}
The matrix of $L_x$ in the standard basis is
\[
[L_x]=
\begin{pmatrix}
a_1 x_1 + a_2 x_2 & a_2 x_1 + a_3 x_2 \\
b_1 x_1 + b_2 x_2 & b_2 x_1 + b_3 x_2
\end{pmatrix}.
\]
Define the \emph{perplex norm} by
\[
N(x):=\det L_x
=(a_1 b_2 - a_2 b_1)\,x_1^2 +(a_1 b_3 - a_3 b_1)\,x_1x_2 -(a_1 a_3 - a_2^2)\,x_2^2.
\]
Then $N$ is a quadratic form with $N(x*y)=N(x)\,N(y)$ for all $x,y$, so $x\in\PP^\times$ iff $N(x)\neq 0$. Geometrically, $|N(x)|$ is the area distortion of $L_x$ and $\mathrm{sign}(N(x))$ records the orientation change.

This quadratic form is naturally related to an involution on $\PP$, which we now describe.

\medskip
\noindent\textbf{Perplex conjugation.}
Define the \emph{perplex conjugate} of $x \in \PP$ by
\[
\widetilde{x}:=\operatorname{adj}(L_x)\,\mathds{1}.
\]
A direct calculation yields the explicit linear formula
\[
\widetilde{x}
= \frac{1}{a_1 b_2 - a_2 b_1}
\begin{bmatrix}
(b_2^2 + a_2 b_1)\, x_1 + (b_2 b_3 + a_3 b_1)\, x_2 \\
- (a_1 b_1 + b_1 b_2)\, x_1 - (b_2^2 + a_2 b_1)\, x_2
\end{bmatrix}.
\]

\begin{prop}
For every $x\in\PP$ one has $x*\widetilde{x}=N(x)\,\mathds{1}$.
\end{prop}

\begin{proof}
Use $L_x\,\operatorname{adj}(L_x)=\det(L_x)\,I_2$.
\end{proof}

The interplay between perplex norm and perplex conjugation provides a functional control on inverses.  
To quantify this control, we next compare multiplication in $\PP$ with the standard norms on~$\R^2$.

\medskip
\noindent\textbf{Quantitative bounds.}
Let $\|x\|_m:=\max\{|x_1|,|x_2|\}$ and let $\|\cdot\|$ be the Euclidean norm.

\begin{prop}\label{prop_norm}
There exists
\[
K:=4\max\{|a_1|,|a_2|,|a_3|,|b_1|,|b_2|,|b_3|\}
\]
such that for all $x,y\in\PP$,
\[
\|x*y\|_m \le K\,\|x\|_m\,\|y\|_m
\qquad\text{and}\qquad
\|x*y\| \le \sqrt{2}\,K\,\|x\|\,\|y\|.
\]
\end{prop}

\begin{lemma}\label{lemma_norm_limits}
Let $\theta>0$ and $N \ge 2$. 
Assume that for a given perplex algebra $\PP$, the homogeneous maps 
\[
F_k:\PP \to \PP, \qquad F_k(t):=t^k \qquad (k=N-1,N),
\]
satisfy $F_k(\hat t)\neq 0$ for all $\hat t \in S_m := \{t \in \PP : \|t\|_m=1\}$. 
Define, for $t\in\PP\setminus\{0\}$,
\[
Q_N(t)=\frac{\|t^N\|_m^\theta}{\|t^{N-1}\|_m}.
\]
Then $\displaystyle\lim_{t\to 0} Q_N(t)=+\infty$.
\end{lemma}

\begin{proof}
Write $t=r\,\hat t$ with $r=\|t\|_m$ and $\hat t\in S_m$.  
By homogeneity,
\[
\|t^k\|_m = \|F_k(r\hat t)\|_m = r^k \,\|F_k(\hat t)\|_m.
\]
The assumption implies that $\|F_k(\hat t)\|_m$ is continuous and nowhere vanishing on the compact set $S_m$, hence there exist constants $0<m_k\le M_k<\infty$ such that
\[
m_k \le \|F_k(\hat t)\|_m \le M_k, \qquad k=N-1,N.
\]
Thus
\[
Q_N(t)\;\ge\;\frac{(m_N\,r^N)^\theta}{M_{N-1}\,r^{N-1}}
=\frac{m_N^\theta}{M_{N-1}}\;r^{\,\theta N - (N-1)}.
\]
If $N>\tfrac{1}{1-\theta}$ then $\theta N - (N-1)<0$, so the right-hand side diverges to $+\infty$ as $r\to 0$.  
\end{proof}

\medskip
\noindent\textbf{Directions and zero divisors.}
Let $Z:=\{z\in\PP\setminus\{0\}:N(z)=0\}$ be the cone of zero divisors and $[Z]\subset\mathbf{P}^1(\R)$ the corresponding set of directions. A sequence $(u_n)\subset\PP^\times$ with $u_n\to 0$ and unit representatives $\hat u_n:=u_n/\|u_n\|$ is said to be \emph{positively separated from $[Z]$} if there exists $c>0$ with $|N(\hat u_n)|\ge c$ for all $n$.

\begin{prop}\label{prop_direction}
If $(u_n)\subset\PP^\times$ with $u_n\to 0$ is positively separated from $[Z]$, then there exists $M>0$ such that
\[
\|u_n^{-1}\| \;\le\; \frac{M}{\|u_n\|}\qquad\text{for all $n$.}
\]
\end{prop}

\begin{proof}
Write $u_n=\|u_n\|\,\hat u_n$ with $\hat u_n\in\mathbb{S}^1$. Then
$\|u_n^{-1}\|=\|u_n\|^{-1}\,\|\hat u_n^{-1}\|$.
By definition, positive separation means $|N(\hat u_n)|=\big|\det L_{\hat u_n}\big|\ge c>0$ for all $n$. 
The map $v\mapsto \operatorname{adj}(L_v)\,\mathds{1}$ is continuous on $\mathbb{S}^1$, hence bounded on the (compact) set $\{\hat u_n\}$: there exists $A>0$ with $\|\operatorname{adj}(L_{\hat u_n})\,\mathds{1}\|\le A$. Using
\[
\hat u_n^{-1}=\frac{\operatorname{adj}(L_{\hat u_n})\,\mathds{1}}{\det L_{\hat u_n}},
\]
we obtain $\|\hat u_n^{-1}\|\le A/c=:M_0$. Therefore
\[
\|u_n^{-1}\|=\frac{\|\hat u_n^{-1}\|}{\|u_n\|}\le \frac{M_0}{\|u_n\|},
\]
which proves the claim with $M:=M_0$.
\end{proof}

\begin{cor}\label{cor_direction_coords}
Under the assumptions of Proposition~\ref{prop_direction}, there exists $C>0$ such that
\[
|(u_n)_i\,(u_n^{-1})_j|\;\le\;C\qquad\text{for all $n$ and }i,j\in\{1,2\}.
\]
\end{cor}

We conclude this section with a concrete illustration, showing both how the classical hyperbolic numbers fit into our framework and how sequences of units can degenerate when approaching zero-divisor directions.

\medskip
\noindent\textbf{Example (hyperbolic numbers).}
For $a=(1,0,1)$ and $b=(0,1,0)$, $\PP$ is the algebra of \emph{hyperbolic} (split-complex) numbers (see \cite{Kantor1989} for instance), with
\[
N(x)=x_1^2-x_2^2,\qquad [Z]=\{[(1,1)],\,[(1,-1)]\}.
\]
Let $x(t)=(t,\,t(1-t))$, $t>0$. Then $x(t)\to 0$ and $N(x(t))=t^3(2-t)\neq 0$ for $0<t<2$, but $x_2(t)/x_1(t)\to 1$, so directions approach $[Z]$ and, writing $y(t)=x(t)^{-1}$, one finds $x_1(t)\,y_1(t)=1/(t(2-t))\to\infty$ as $t\to 0^+$.

%%%%%%%%%%%%%%%%%%%%%%%%%%%%%%%%%%%%%%%%%%%%%%%%%%%%%%%%%%
%%%%%%%%%%%%%%%%%%%%%%%%%%%%%%%%%%%%%%%%%%%%%%%%%%%%%%%%%%
\medskip
\section{Differentiability in perplex algebras}\label{sec:diff}

We now develop the differential calculus in the setting of perplex algebras.  
The goal is to extend the basic principles of holomorphic function theory to this broader framework, while highlighting the new analytic phenomena that arise in the real context.  
We introduce the notion of $\PP$--differentiability, defined through increments along units separated from zero divisors, and establish the fundamental rules of calculus in this setting.  
Polynomial maps are shown to be $\PP$--differentiable with the expected formulas for their derivatives.  
We then derive the generalized Cauchy--Riemann equations, which characterize differentiability in terms of real partial derivatives and recover the classical complex and hyperbolic cases.  
Finally, we relate the $\PP$--derivative to the real differential, obtaining a natural description of critical points of perplex functions in direct analogy with the holomorphic case.

\medskip
\subsection{Derivative} 

Let $\PP$ be a perplex algebra as above and let $f: \PP \to \PP$ be a function.

\begin{defi} \label{defi_diff}
We say that $f$ is \emph{$\PP$-differentiable} at a point $x \in \PP$ if there exists $L \in \PP$ such that, for every positively separated sequence $(h_n) \subset \PP^\times$ (see Section~\ref{sec:perplex}) converging to the origin, one has
\[
\lim_{n \to \infty} \big( f(x+ h_n) - f(x) \big) * h_n^{-1} = L,
\]
and, as a function of $x$, $L$ is continuous on some neighborhood of $x$.  
In this case, we call $L$ the \emph{derivative} of $f$ at $x$ and denote it by $f'(x)$.  
We say that $f$ is \emph{$\PP$-differentiable} if it is $\PP$-differentiable at every $x \in \PP$.  
\end{defi}

Thus $\PP$–differentiability is defined by testing increments along units separated from zero divisors, with the additional requirement that the resulting derivative depends continuously on the base point.

Whenever there is no risk of misinterpretation, we will simply say that $f$ is differentiable.

It is easy to verify that the usual differentiation rules hold in this setting. Precisely, we have:

\begin{prop} \label{diff_rules}
Let $f,g : \mathbb{P}_{a,b} \to \mathbb{P}_{a,b}$ be differentiable at $x \in \mathbb{P}_{a,b}$.  
Let $c \in \mathbb{P}_{a,b}$ and $\alpha \in \mathbb{R}$. Then:
\begin{itemize}
    \item[$(i)$] The sum $f+g$ is differentiable at $x$ and
    \[
    (f+g)'(x) = f'(x) + g'(x) \, .
    \]
    \item[$(ii)$] The product $f*g$ is differentiable at $x$ and
    \[
    (f*g)'(x) = \big(f'(x) * g(x)\big) + \big(f(x) * g'(x)\big) \, .
    \]
    \item[$(iii)$] The product $c * f$ is differentiable at $x$ and
    \[
    (c * f)'(x) = c * f'(x) \, .
    \]
    \item[$(iv)$] The scalar product $\alpha f$ is differentiable at $x$ and
    \[
    (\alpha f)'(x) = \alpha\, f'(x) \, .
    \]
\end{itemize}
\end{prop}

As an immediate consequence, we have:

\begin{cor}
Let $f \in \PP[x]$ be the polynomial function
\[
f(x) = a_n * x^n + \dots + a_2 * x^2 + a_1 * x + a_0,
\]
with $a_n \neq \begin{bmatrix} 0 \\ 0 \end{bmatrix}$. Then $f$ is differentiable and
\[
f'(x) = n\,a_n * x^{n-1} + \dots + 2\,a_2 * x + a_1 \, .
\]
\end{cor}

\medskip
\subsection{Generalized Cauchy--Riemann equation}

Let $f : \PP \to \PP$ be a differentiable function, and let 
\[
f(x_1,x_2) = \big( u(x_1,x_2),\, v(x_1,x_2) \big)
\]
be its coordinate representation as a map $\mathbb{R}^2 \to \mathbb{R}^2$.  
Assume also that $f$ is differentiable in the usual real sense, and denote by $u_{x_i}$ and $v_{x_i}$ the corresponding partial derivatives with respect to $x_i$ ($i=1,2$).  

\begin{prop} \label{prop_equations}
If $f : \PP \to \PP$ is differentiable, then for every $x\in\PP$ one has:
\begin{equation} \label{eq_CR2}
f'(x) = \begin{bmatrix} 0 \\ 1 \end{bmatrix}^{-1} * 
\begin{bmatrix} u_{x_2} \\ v_{x_2} \end{bmatrix} ,
\end{equation}
\begin{equation} \label{eq_CR1}
f'(x) = \begin{bmatrix} 1 \\ 0 \end{bmatrix}^{-1} * 
\begin{bmatrix} u_{x_1} \\ v_{x_1} \end{bmatrix} .
\end{equation}
\end{prop}

These formulas identify the derivative with left multiplication by a vector involving the real partial derivatives, and prepare the way for the generalized Cauchy–Riemann condition.

\begin{proof}
Take $h_r = r\begin{bmatrix}0\\1\end{bmatrix}$, which is a unit in $\PP$.  
Then, by Definition~\ref{defi_diff},
\[
f'(x) = \lim_{r\to 0} \frac{f(x_1, x_2+r)-f(x_1,x_2)}{r} * \begin{bmatrix}0\\1\end{bmatrix}^{-1}
= \begin{bmatrix} u_{x_2} \\ v_{x_2} \end{bmatrix} * \begin{bmatrix}0\\1\end{bmatrix}^{-1},
\]
proving \eqref{eq_CR2}.  
The identity \eqref{eq_CR1} follows analogously with $h_r=r\begin{bmatrix}1\\0\end{bmatrix}$.
\end{proof}

Thus, combining \eqref{eq_CR2} and \eqref{eq_CR1} we obtain:

\begin{cor} \label{cor_CR}
If $f : \PP \to \PP$ is differentiable, then
\begin{equation*} 
\begin{bmatrix}
0 \\ 
1  
\end{bmatrix} *
\begin{bmatrix}
u_{x_1} \\ 
v_{x_1}  
\end{bmatrix}
\;=\;
\begin{bmatrix}
1 \\ 
0  
\end{bmatrix} *
\begin{bmatrix}
u_{x_2} \\ 
v_{x_2}  
\end{bmatrix}.
\end{equation*}
\end{cor}

\begin{defi}
We call the identity in Corollary~\ref{cor_CR} the \emph{generalized Cauchy--Riemann equation}.
\end{defi}

Recall from Definition~\ref{defi_diff} that if $f$ is $\PP$-differentiable, then its derivative $f'$ is continuous.  
As an immediate consequence of Proposition~\ref{prop_equations}, we have:

\begin{cor}
If $f : \mathbb{R}^2 \to \mathbb{R}^2$ is $\PP$-differentiable for some $(a,b) \in \mathcal{P}$, then $f$ is differentiable in the usual real sense.
\end{cor}

\begin{remark}
In the special case where $\PP=\mathbb{C}$, the generalized Cauchy--Riemann equation in Corollary~\ref{cor_CR} reduces exactly to the classical Cauchy--Riemann equations
\[
u_{x_1} = v_{x_2}, \qquad u_{x_2} = -v_{x_1}.
\]
Thus, the generalized equation recovers the usual holomorphicity condition in complex analysis.
\end{remark}

\begin{example}[Hyperbolic case]
Consider $\PP$ with parameters $a = (1,0,1)$ and $b = (0,1,0)$. If $f = (u,v)$ is $\PP$-differentiable, Corollary~\ref{cor_CR} becomes
\[
\begin{bmatrix} 0 \\ 1 \end{bmatrix} *
\begin{bmatrix} u_{x_1} \\ v_{x_1} \end{bmatrix}
=
\begin{bmatrix} 1 \\ 0 \end{bmatrix} *
\begin{bmatrix} u_{x_2} \\ v_{x_2} \end{bmatrix},
\]
which, using the multiplication rule above, is equivalent to the linear system
\[
\begin{cases}
u_{x_1} = v_{x_2}, \\
u_{x_2} = v_{x_1}.
\end{cases}
\]
\end{example}

The results so far show that $\PP$--differentiability implies the genera\-lized Cauchy--Riemann equation.  
Conversely, we now prove that this condition is also sufficient, giving a complete characterization of $\PP$--differentiability.

\begin{theo} \label{theo_GCR}
Let $f:\R^2\to\R^2$ be a $C^1$ map and let $(a,b)\in\mathcal P$. Then $f$ is $\PP$-differentiable if and only if the corresponding generalized Cauchy--Riemann equation holds.
\end{theo}

\begin{proof}
Assume $f=(u,v)$ is $C^1$ and satisfies the generalized Cauchy--Riemann equation. Define, for each $x$, the continuous function
\[
w(x)\;:=\;\begin{bmatrix}1\\0\end{bmatrix}^{-1} * \begin{bmatrix} u_{x_1}(x)\\ v_{x_1}(x)\end{bmatrix}
\;=\;\begin{bmatrix}0\\1\end{bmatrix}^{-1} * \begin{bmatrix} u_{x_2}(x)\\ v_{x_2}(x)\end{bmatrix}.
\]

\smallskip
\emph{Step 1: $Df(x)=L_{w(x)}$ as linear maps $\R^2\to\R^2$.}
Since $f$ is $C^1$, its real differential $Df(x)$ satisfies
\[
Df(x)\,e_i=\begin{bmatrix}u_{x_i}(x)\\ v_{x_i}(x)\end{bmatrix}, \qquad i=1,2,
\]
where $e_1=\begin{bmatrix}1\\0\end{bmatrix}$ and $e_2=\begin{bmatrix}0\\1\end{bmatrix}$. By the definition of $w(x)$,
\[
\begin{bmatrix}u_{x_1}\\ v_{x_1}\end{bmatrix}=e_1 * w(x), 
\qquad
\begin{bmatrix}u_{x_2}\\ v_{x_2}\end{bmatrix}=e_2 * w(x).
\]
Because $*$ is commutative, $L_{w(x)}(e_i)=w(x)*e_i=e_i*w(x)$, hence
\[
Df(x)\,e_i = L_{w(x)}(e_i)\quad (i=1,2).
\]
Therefore $Df(x)=L_{w(x)}$.

\smallskip
\emph{Step 2: Limit along positively separated sequences.}
Let $(h_n)\subset\PP^\times$ be any positively separated sequence with $h_n\to 0$.  
As a real map,
\[
f(x+h_n)-f(x) = Df(x)\,h_n + r(h_n) = L_{w(x)}(h_n) + r(h_n),
\]
where $r(h)$ is the remainder term in the first-order expansion, satisfying
\[
\frac{\|r(h)\|}{\|h\|} \to 0 \quad \text{as} \quad h \to 0.
\]
Hence
\[
\big(f(x+h_n)-f(x)\big)*h_n^{-1} 
= \big(L_{w(x)} h_n\big)*h_n^{-1} \,+\, r(h_n)*h_n^{-1}
= w(x) + r(h_n)*h_n^{-1},
\]
using associativity and $L_{w(x)}h_n=w(x)*h_n$.

By Proposition~\ref{prop_norm}, there exists $C_* > 0$ such that
\[
\|r(h_n)*h_n^{-1}\| \;\le\; C_* \,\|r(h_n)\|\,\|h_n^{-1}\|.
\]
Write $h_n = r_n u_n$ with $r_n\in\R^\times$ and $\|u_n\|=1$.  
Positive separation from $[Z]$ implies $\|u_n^{-1}\|\le M$ for some $M>0$ (see Proposition \ref{prop_direction}). Then
\[
\|h_n^{-1}\| = \frac{\|u_n^{-1}\|}{|r_n|} \le \frac{M}{\|h_n\|}.
\]
Therefore,
\[
\|r(h_n)*h_n^{-1}\| \le \frac{C_* M \,\|r(h_n)\|}{\|h_n\|} \to 0.
\]

\smallskip
We conclude that
\[
\lim_{n\to\infty} \big(f(x+h_n)-f(x)\big)*h_n^{-1} = w(x),
\]
independently of the positively separated sequence $(h_n)$.  
Since $w$ is continuous, $f$ is $\PP$-differentiable at $x$ with $f'(x)=w(x)$. This completes the proof.
\end{proof}

\begin{example} \label{example_complex_conjugate}
The complex conjugation map 
\[
f(x_1,x_2) = (x_1,\,-x_2)
\]
is not $\PP$-differentiable for any $(a,b) \in \mathcal{P}$, since the generalized Cauchy--Riemann equation fails to hold in every case.
\end{example}

As consequences of Theorem \ref{theo_GCR} we have:

\begin{cor}
If $f: \PP \to \PP$ is differentiable and $f'(x) = 0$ for every $x$ in some open neighborhood $U \subset \PP$, then $f$ is constant on $U$.
\end{cor}

\begin{cor} \label{cor_cp}
    Let $f: \PP \to \PP$ be a differentiable function.  
    Then $x$ is a critical point of $f$ (in the usual sense, as a map $\R^2 \to \R^2$) if and only if $f'(x) \in \PP \setminus \PP^\times$.
\end{cor}

\begin{proof}
Recall that $y \in \PP^\times$ if and only if its perplex norm $N(y) \neq 0$, which is also equivalent to $L_y$ being invertible, where $L_{y}$ denotes left multiplication by $y \in \PP$. By Theorem~\ref{theo_GCR}, if $f$ is $\PP$-differentiable then its real differential satisfies
\[
Df(x) = L_{f'(x)} : \R^2 \to \R^2.
\]
In particular, $Df(x)$ is invertible if and only if $L_{f'(x)}$ is invertible. Since $L_{y}$ is invertible if and only if $y \in \PP^\times$, we have:
\[
Df(x) \ \text{invertible} \quad \Longleftrightarrow \quad f'(x) \in \PP^\times.
\]
Therefore $Df(x)$ fails to be invertible if and only if $f'(x) \in \PP \setminus \PP^\times$. This proves the claim.
\end{proof}

This provides a natural extension of the notion of critical point from holomorphic maps to the wider perplex analytic setting.

%%%%%%%%%%%%%%%%%%%%%%%%%%%%%%%%%%%%%%%%%%%%%%%%%%%%%%%%%%
%%%%%%%%%%%%%%%%%%%%%%%%%%%%%%%%%%%%%%%%%%%%%%%%%%%%%%%%%%
\section{Perplex functions and approximation}\label{sec:approx}

In this section, we investigate the class of maps $\R^2 \to \R^2$ that are $\PP$--differentiable for some $(a,b)\in\mathcal P$, which we call \emph{perplex functions}. They are the natural analogue of holomorphic maps in this framework. A key feature is their rigidity: for a fixed algebra $\PP$, the generalized Cauchy--Riemann (GCR) equation severely restricts the class of functions.  
In particular, it is not dense among polynomial maps.  
However, if $(a,b)$ is allowed to vary, a rich approximation theory emerges in the linear and quadratic cases.

\begin{defi}
A $C^1$ map $f:\R^2\to\R^2$ is a \emph{perplex function} if there exists $(a,b)\in\mathcal P$ such that $f$ is $\PP$--differentiable.
\end{defi}

Equivalently, $f=(u,v)$ is a perplex function if and only if its derivatives satisfy
\begin{equation}\label{eq:matrix_GCR}
\begin{pmatrix} a_2 & a_3 \\ b_2 & b_3 \end{pmatrix}
\binom{u_{x_1}}{v_{x_1}}
=
\begin{pmatrix} a_1 & a_2 \\ b_1 & b_2 \end{pmatrix}
\binom{u_{x_2}}{v_{x_2}} 
\end{equation}
for some $(a,b)\in\mathcal P$. Thus $\PP$--differentiability reduces to linear relations among partial derivatives.  
This rigidity prevents approximation when $(a,b)$ is fixed.

\begin{example}\label{example_approx1}
Let $f(x_1,x_2)=(x_1,-x_2)$.  
Perturbations $f_{\alpha,\beta}(x_1,x_2)=(x_1+\alpha x_2,\;\beta x_1-x_2)$ satisfy \eqref{eq:matrix_GCR} only for special choices of $(\alpha,\beta)$ depending on $(a,b)$.  
If $(a,b)$ is fixed, no sequence with $\alpha,\beta\to 0$ can make $f_{\alpha,\beta}$ $\PP$--differentiable, since $f$ itself is not (Example~\ref{example_complex_conjugate}).  
Hence polynomial maps cannot in general be approximated by $\PP$--functions with $(a,b)$ fixed.
\end{example}

Allowing $(a,b)$ to vary along the approximating sequence remedies this obstruction.  

\begin{theo}\label{thm:linear_perplex}
Let $L:\R^2\to\R^2$ be linear with matrix $J=\begin{pmatrix}p&r\\ q&s\end{pmatrix}$. Then:
\begin{itemize}
\item[(i)] For a Zariski-open dense set of $J$, there exists $(a,b)\in\mathcal P$ making $L$ $\PP$--differentiable.
\item[(ii)] For any $J$, there exist $(a_n,b_n)\in\mathcal P$ and $\mathbb{P}_{a_n,b_n}$--differentiable linear maps $L_n$ with $L_n\to L$ in the Whitney topology.  
Hence every linear map is a limit of polynomial perplex functions.
\end{itemize}
\end{theo}

\begin{proof}
For $L=(u,v)$, the GCR equation becomes
\[
a_2 p+a_3 q=a_1 r+a_2 s,\qquad b_2 p+b_3 q=b_1 r+b_2 s,
\]
together with the defining relations of $\mathcal P$.  
This system has solutions for generic $(p,q,r,s)$ by dimension count, yielding (i).  
For (ii), approximate any $J$ by a sequence $J_n$ in the generic set and solve for $(a_n,b_n)$; the corresponding maps $L_n$ converge to $L$.
\end{proof}

\begin{example}\label{example_conj}
The conjugation map $f(x_1,x_2)=(x_1,-x_2)$ is not $\PP$--differentiable for any fixed $(a,b)$.  
But Theorem~\ref{thm:linear_perplex}(ii) provides sequences $(a_n,b_n)$ and polynomial maps $f_n$ with $f_n\to f$.  
Thus the obstruction in Example~\ref{example_approx1} disappears once $(a,b)$ is allowed to vary.
\end{example}

Together, Examples~\ref{example_approx1} and~\ref{example_conj} show that varying $(a,b)$ is essential for a viable approximation theory.

\medskip
\noindent\textbf{Quadratic maps.}
For quadratic maps, the situation is subtler. Writing
\[
\binom{u_{x_1}}{v_{x_1}}=m_0+m_1 x_1+m_2 x_2,\qquad
\binom{u_{x_2}}{v_{x_2}}=n_0+n_1 x_1+n_2 x_2,
\]
with $m_k,n_k\in\R^2$, the GCR equations yield:

\begin{prop}[Quadratic characterization]\label{prop:quad_char}
A quadratic map $f$ is a perplex function iff there exists $T\in M_2(\R)$ constant such that
\[
n_k=T\,m_k\qquad(k=0,1,2).
\]
Equivalently, $\binom{u_{x_2}}{v_{x_2}}=T\,\binom{u_{x_1}}{v_{x_1}}$ for all $x$.
\end{prop}

\begin{proof}[Proof of Proposition~\ref{prop:quad_char}]
If $f$ is $\PP$--differentiable, the GCR equation reads 
\[
A\binom{u_{x_1}}{v_{x_1}}=B\binom{u_{x_2}}{v_{x_2}},\qquad
A=\begin{pmatrix}a_2&a_3\\ b_2&b_3\end{pmatrix},\ 
B=\begin{pmatrix}a_1&a_2\\ b_1&b_2\end{pmatrix}.
\]
Since $B$ is invertible, setting $T:=B^{-1}A$ yields 
$\binom{u_{x_2}}{v_{x_2}}=T\binom{u_{x_1}}{v_{x_1}}$, hence $T m_k=n_k$ for $k=0,1,2$.  
Conversely, given such a constant $T$, choose any $B\in\mathrm{GL}_2(\R)$, set $A:=BT$, and solve the two defining equalities of $\mathcal P$ by adjusting $B$ (the open inequalities can be met by a small perturbation).  
Thus a valid $(a,b)\in\mathcal P$ exists, and Theorem~\ref{theo_GCR} gives the claim.
\end{proof}

\begin{example}
The map $g(x_1,x_2)=(x_1^2,x_2^2)$ does not satisfy the condition of Proposition~\ref{prop:quad_char}, hence is not a perplex function.
\end{example}

\begin{cor}
Quadratic perplex functions are dense among quadratic maps in the Whitney topology.
\end{cor}

\begin{proof}
For generic $f$, $m_0$ and $m_1$ are linearly independent, so $T$ is uniquely determined by $n_0,n_1$.  
Adjusting coefficients slightly allows $n_2=T m_2$ to hold as well.  
Thus any quadratic can be approximated by a perplex quadratic.
\end{proof}

\begin{example}
Approximating $g(x_1,x_2)=(x_1^2,x_2^2)$: set
\[
f_\varepsilon(x_1,x_2)=(x_1^2+\varepsilon x_1x_2,\;\;\varepsilon x_1^2+x_2^2),\qquad \varepsilon\neq 0.
\]
Then $f_\varepsilon\to g$ smoothly as $\varepsilon\to 0$, and each $f_\varepsilon$ is a quadratic perplex function.
\end{example}

%%%%%%%%%%%%%%%%%%%%%%%%%%%%%%%%%%%%%%%%%%%%%%%%%%%%%%%%%%
%%%%%%%%%%%%%%%%%%%%%%%%%%%%%%%%%%%%%%%%%%%%%%%%%%%%%%%%%%
\section{Functions on several perplex variables}\label{sec:several}

We now extend the theory from one variable to several perplex variables.  
The natural ambient space is the free $\PP$--module $\PP^n$, which provides the right framework to define partial derivatives, directional derivatives, and tangent spaces.  
Within this setting we establish the analogue of the Łojasiewicz inequality for analytic perplex functions.

\subsection{Derivatives in several variables}

Fix a perplex algebra $\PP$. For $n\ge1$, write
\[
p=(p_1,\dots,p_n)\in \PP^n,\qquad p_i\in\PP.
\]
The canonical basis is $\{E_1,\dots,E_n\}$, where $E_i$ has entry $\mathds{1}$ in the $i$th coordinate and zero elsewhere.  
Scalar multiplication is defined componentwise: $x*p:=(x*p_1,\dots,x*p_n)$.

\begin{defi}\label{defi_multidiff}
Let $f:\PP^n\to\PP$ and $p\in\PP^n$.  
We say that $f$ admits the \emph{partial derivative with respect to $p_i$} at $p$ if there exists $L_i\in\PP$ such that, for every positively separated sequence $(h_n)\subset\PP^\times$ with $h_n\to0$,
\[
\lim_{n\to\infty} \big(f(p+h_n*E_i)-f(p)\big)*h_n^{-1}=L_i,
\]
and $L_i$ depends continuously on $p$. We then set $\tfrac{\partial f}{\partial p_i}(p):=L_i$.  
If this holds for all $i$, we say that $f$ is \emph{$\PP$--differentiable at $p$}.
\end{defi}

Writing $f=(u,v):\R^{2n}\to\R^2$, Theorem~\ref{theo_GCR} yields the multivariable generalized Cauchy--Riemann (GCR) equations:  

\begin{prop}\label{prop_multiGCR}
The map $f$ is $\PP$--differentiable at $p$ if and only if, for each $i=1,\dots,n$,
\begin{equation}\label{eq:multiGCR}
\begin{pmatrix} a_2 & a_3 \\ b_2 & b_3 \end{pmatrix}
\binom{u_{x_{i1}}}{v_{x_{i1}}}
=
\begin{pmatrix} a_1 & a_2 \\ b_1 & b_2 \end{pmatrix}
\binom{u_{x_{i2}}}{v_{x_{i2}}}.
\end{equation}
\end{prop}

Given $w=\sum_i w_i*E_i\in\PP^n$, the directional derivative is
\[
df_p\cdot w=\sum_{i=1}^n w_i * \tfrac{\partial f}{\partial p_i}(p).
\]
The \emph{perplex gradient} is defined by
\[
\nabla f(p)=\Big(\tfrac{\partial f}{\partial p_1}(p),\dots,\tfrac{\partial f}{\partial p_n}(p)\Big).
\]

\begin{lemma}\label{lemma_Dfp}
Let $Df_p:\R^{2n}\to\R^2$ be the real Jacobian of $f$ at $p$.  
If $w_i=\binom{w_{i1}}{w_{i2}}\in\R^2$, then
\[
df_p\cdot (w_1,\dots,w_n)=Df_p\cdot (w_{11},w_{12},\dots,w_{n1},w_{n2}).
\]
\end{lemma}

\begin{proof}
By \eqref{eq:multiGCR}, $\binom{u_{x_{i1}}}{v_{x_{i1}}}=e_1*\tfrac{\partial f}{\partial p_i}(p)$ and $\binom{u_{x_{i2}}}{v_{x_{i2}}}=e_2*\tfrac{\partial f}{\partial p_i}(p)$.  
Hence
\[
Df_p(w_{11},\dots,w_{n2})=\sum_{i=1}^n(w_{i1}e_1+w_{i2}e_2)*\tfrac{\partial f}{\partial p_i}(p)=df_p\cdot(w_1,\dots,w_n).
\]
\end{proof}

Set
\[
J_p := \big\langle \tfrac{\partial f}{\partial p_1}(p),\dots,\tfrac{\partial f}{\partial p_n}(p)\big\rangle
\subset \PP,
\]
that is, the ideal of $\PP$ generated by the partial derivatives of $f$ at $p$.

\begin{lemma}
Let $f:\PP^n\to\PP$ be $\PP$-differentiable and $p\in\PP^n$. Then the following are equivalent:
\begin{enumerate}
\item $p$ is a critical point ($\operatorname{rank}Df_p<2$);
\item $J_p$ is a proper ideal of $\PP$;
\item $df_p(\PP^n)\subset \PP\setminus\PP^\times$.
\end{enumerate}
\end{lemma}

\begin{defi}
A \emph{perplex hyperplane} in $\PP^n$ is a free $\PP$--submodule of rank $n-1$.
Equivalently, it is the kernel of a nonzero $\PP$--linear functional 
$\ell:\PP^n\to\PP$.  
Under the identification $\PP^n\cong\R^{2n}$, a perplex hyperplane corresponds
to a real linear subspace of codimension $2$.
\end{defi}

\begin{cor}\label{cor_ts}
If $p$ is regular, then $\ker(df_p)$ is a perplex hyperplane in $\PP^n$ and
\[
\ker(df_p)=\ker(Df_p)=T_p\big(f^{-1}(f(p))\big)\subset\R^{2n}.
\]
\end{cor}

\begin{proof}
Since $p$ is regular, $Df_p:\R^{2n}\to\R^2$ has real rank $2$. By Lemma~\ref{lemma_Dfp},
\[
df_p(\PP^n)\;=\;\Big\{\sum_{i=1}^n w_i*\tfrac{\partial f}{\partial p_i}(p):w_i\in\PP\Big\}
=:J_p\subset\PP .
\]
Because $\operatorname{rank} Df_p=2$, the ideal $J_p$ contains a unit $u\in\PP^\times$, hence $df_p$ is surjective as a $\PP$–linear map. Choose $c=(c_1,\dots,c_n)\in\PP^n$ with
$\sum_i c_i*\tfrac{\partial f}{\partial p_i}(p)=u$, and define a $\PP$–linear section
$s:\PP\to\PP^n$, $s(y):=(y*u^{-1})*c$. Then $df_p\circ s=\mathrm{id}_{\PP}$, so
\[
\PP^n \;=\; \ker(df_p)\oplus \mathrm{im}(s)\ \cong\ \ker(df_p)\oplus\PP,
\]
which shows that $\ker(df_p)$ is a free $\PP$–module of rank $n-1$.

The identification $\ker(df_p)=\ker(Df_p)=T_p(f^{-1}(f(p)))$ follows directly from 
Lemma~\ref{lemma_Dfp} and the implicit function theorem.
\end{proof}

Finally, the following technical lemma will be needed later.

\begin{lemma}\label{lemma_limites}
Let $f:\PP^{n+1}\to\PP$ be $\PP$--differentiable and let $(p_k)\to p$ be a sequence of regular points.  
Write the gradient at $p_k$ as
\[
\nabla f(p_k)=\Big(\tfrac{\partial f}{\partial p_1}(p_k),\dots,\tfrac{\partial f}{\partial p_{n+1}}(p_k)\Big)\in\PP^{\,n+1}.
\]
Decompose $\nabla f(p_k)=(\alpha_k,\beta_k)$ with $\alpha_k\in\PP^n$ (the first $n$ components) 
and $\beta_k\in\PP$ (the last component).  
If
\[
\frac{\|\alpha_k\|_m}{\|\beta_k\|_m}\;\longrightarrow\;\infty,
\]
then the hyperplanes $\ker Df_{p_k}$ do not converge to $\PP^n\times\{0\}$.
\end{lemma}

\begin{proof}
Each kernel $\ker Df_{p_k}$ is the hyperplane defined by the linear functional 
\[
L_k(u_1,\dots,u_{n+1})=\sum_{i=1}^{n} u_i * (\alpha_k)_i \;+\; u_{n+1} * \beta_k.
\]
Normalizing, set $\tilde c_k=(\tilde\alpha_k,\tilde\beta_k):=(\alpha_k,\beta_k)/\|(\alpha_k,\beta_k)\|_m$.  
The hypothesis $\|\alpha_k\|_m/\|\beta_k\|_m\to\infty$ implies 
$\tilde\alpha_k\to\alpha_\infty\neq 0$ and $\tilde\beta_k\to 0$.  
Thus the defining functionals $\tilde L_k$ converge to 
\[
L_\infty(u_1,\dots,u_{n+1})=\sum_{i=1}^n u_i * (\alpha_\infty)_i,
\]
whose kernel is not $\PP^n\times\{0\}$.  
Hence $\ker Df_{p_k}$ cannot converge to $\PP^n\times\{0\}$.
\end{proof}

\subsection{The Łojasiewicz inequality}

For a real analytic $f:\R^n\to\R$ with $f(0)=0$, the classical Łojasiewicz inequality states that there exist $C>0$ and $0<\theta<1$ such that $\|\nabla f(x)\|\ge C|f(x)|^\theta$ near $0$ \cite{Lo}.  
Analogues hold in the complex case. We now extend this to analytic perplex functions.

\begin{defi}
A map $f:\PP^n\to\PP$ is \emph{$\PP$--analytic} if it is $\PP$--differentiable and its real coordinate functions $u,v:\R^{2n}\to\R$ are real analytic.  
A \emph{perplex analytic function} is a real analytic map $f:\R^{2n}\to\R^2$ that is $\PP$--analytic for some $(a,b)\in\mathcal P$.
\end{defi}

\begin{theo}\label{theo_Loja}
Let $f:\PP^n\to\PP$ be $\PP$--analytic with $f(0)=0$.  
Then there exist $U\ni0$, $C>0$, and $0<\theta<1$ such that
\[
\|\nabla f(p)\|_m\;\ge\; C\,\|f(p)\|_m^\theta,\qquad p\in U.
\]
\end{theo}

\begin{proof}
For a $\PP$--differentiable function $g:\PP\to\PP$, the GCR equation and Proposition~\ref{prop_norm} imply
\[
\|g'(x)\|_m\;\ge\; \tfrac1K \max\{|u_{x_1}(x)|,|u_{x_2}(x)|,|v_{x_1}(x)|,|v_{x_2}(x)|\},\quad K>0.
\]
Hence
\[
\|\nabla f(p)\|_m \;\ge\; \tfrac1{\sqrt{2}K}\max\{\|\nabla u(p)\|,\|\nabla v(p)\|\}.
\]
Applying the classical Łojasiewicz inequality to $u$ and $v$ gives
\[
\|\nabla f(p)\|_m \;\ge\; C\|f(p)\|_m^\theta
\]
near $0$.
\end{proof}

\begin{cor}
If $f:\R^{2n}\to\R^2$ is a perplex analytic function with $f(0)=0$, then in a neighborhood of $0$ every $p\notin f^{-1}(0)$ satisfies $\operatorname{rank}Df_p\ge1$ (that is, the real differential $Df_p$ is nonzero).
\end{cor}

\begin{remark}
Unlike the complex case, $\operatorname{rank}Df_p$ may equal $1$ rather than $2$, since nonzero partial derivatives need not be units in $\PP$. This phenomenon reflects the new algebraic structure of the perplex setting.
\end{remark}

%%%%%%%%%%%%%%%%%%%%%%%%%%%%%%%%%%%%%%%%%%%%%%%%%%%%%%%%%%
%%%%%%%%%%%%%%%%%%%%%%%%%%%%%%%%%%%%%%%%%%%%%%%%%%%%%%%%%%
\section{Singularities in the perplex setting}\label{sec:sing}

In the classical theory of complex analytic functions, two fundamental results describe the local topology of singularities: the existence of a good stratification and the Milnor--Lê fibration theorem.  
In this section we show that these extend naturally to the perplex setting.  
Using the Łojasiewicz inequality from Theorem~\ref{theo_Loja}, we adapt the arguments of Hamm and Lê to prove that every $\PP$--analytic function admits a good stratification, whenever $\PP$ is nondegenerate.  
As a consequence, we obtain a Milnor--Lê type fibration theorem for nondegenerate perplex analytic maps $f:\R^{2n}\to\R^2$, giving the basic local topological structure around their singularities.

\subsection{Good stratification}

Let $f:\PP^n\to\PP$ be $\PP$--analytic with $f(0)=0$, and set
$V:=\{f=0\}\subset\R^{2n}$.  
Let $\mathcal S=\{\mathcal S_\alpha\}_{\alpha\in\Lambda}$ be a Whitney stratification of $V$.

Following \cite{HL}, we say that $\mathcal S$ is a \emph{good stratification} of $V$ at $0$ if there exists a neighborhood $U\subset\R^{2n}$ of $0$ such that, for every sequence $q_i\in U\setminus V$ converging to $q\in V\cap U$, with tangent spaces 
$T_{q_i}(f^{-1}(f(q_i)))$ converging to a plane $T$, one has
\[
T \supset T_q\mathcal S_{\alpha(q)},
\]
where $\mathcal S_{\alpha(q)}$ is the stratum containing $q$.

Hamm and Lê proved in \cite{HL} that every $\C$--analytic function admits a good stratification.  
Since their proof relies only on the Łojasiewicz inequality, it extends verbatim to the perplex case via Theorem~\ref{theo_Loja}:

\begin{theo}\label{theo_stratification}
Let $\PP$ be a nondegenerate perplex algebra. 
Then every $\PP$–analytic function $f : \PP^n \to \PP$
admits a good stratification.
\end{theo}

\begin{proof}
Let $f:\PP^n\to\PP$ be $\PP$--analytic with $f(0)=0$ and set
$V:=\{f=0\}\subset\R^{2n}$.
Fix $N\in\mathbb{N}$ even and consider the $\PP$--analytic map
\[
g_N:\PP^n\times\PP\longrightarrow\PP,\qquad g_N(p,t):=f(p)-t^N,
\]
and the real-analytic hypersurface
\[
G_N:=\{g_N=0\}\cap (U\times\PP),
\]
where $U\subset\PP^n$ is a sufficiently small neighborhood of $0$ (to be chosen later).
Since $G_N$ is a closed real-analytic subset of $U\times\PP\cong\R^{2(n+1)}$,
it admits a Whitney stratification $\mathcal G^N$ such that $V\times\{0\}$ is a union of strata.
Projecting to the first factor, the stratification $\mathcal G^N$
induces a Whitney stratification $\mathcal S^N$ of $V$. We claim that for $N$ large enough, $\mathcal S^N$ is a good stratification of $V$ at $0$.

Suppose not. Then there exist $\varepsilon>0$ and a sequence
$q_i\in \big(\B_\varepsilon\cap U\big)\setminus V$ converging to $q\in V\cap \B_\varepsilon$ such that:
\begin{itemize}
\item[(a)] each $q_i$ is a regular point of $f$ (i.e. $\operatorname{rank}Df_{q_i}=2$);
\item[(b)] the tangent spaces $T_{q_i}\big(f^{-1}(f(q_i))\big)$ are defined and converge to a hyperplane
$T\subset \PP^n$;
\item[(c)] $T$ does not contain $T_q \mathcal S^N_{\alpha(q)}$ (the tangent space to the stratum through $q$).
\end{itemize}
Since $N$ is even, one can set $t_i\in\PP$ by $t_i^N:=f(q_i)$. Set $x_i:=(q_i,t_i)\in G_N$; then $x_i\to x:=(q,0)\in V\times\{0\}$.
By Corollary~\ref{cor_ts}, the real tangent space of the level set satisfies
\[
T_{q_i}\big(f^{-1}(f(q_i))\big)\;=\;\ker Df_{q_i}\subset \PP^n,
\]
and since $\operatorname{rank}Df_{q_i}=2$, each $\ker Df_{q_i}$ is a perplex hyperplane.
Likewise, because $q_i$ is regular for $f$, the differential of $g_N$ at $x_i$,
\[
D(g_N)_{x_i}:(\R^{2n}\times \R^2)\longrightarrow \R^2,
\qquad
D(g_N)_{x_i}(v,s)\;=\; Df_{q_i}(v) \;-\; N\,t_i^{N-1} * s,
\]
has real rank $2$, hence by Corollary~\ref{cor_ts} applied to $g_N$ the tangent space
\[
T_{x_i}G_N=\ker D(g_N)_{x_i}
\]
is a perplex hyperplane in $\PP^{n+1}$.
Moreover,
\[
T_{x_i}\big(f^{-1}(f(q_i))\times\{t_i\}\big)
=
\ker Df_{q_i} \times \{0\}
\subset
\ker D(g_N)_{x_i}
=
T_{x_i}G_N.
\]
By Whitney condition (a) for the stratification $\mathcal G^N$ of $G_N$, the limit
\[
\tau\;:=\;\lim_{i\to\infty} T_{x_i}G_N
\]
exists (up to subsequences) and contains the tangent space
$T_x \mathcal G^N_{\alpha(x)}$ of the stratum of $G_N$ through $x=(q,0)$.
By construction of $\mathcal S^N$, one has
\[
T_x \mathcal G^N_{\alpha(x)}\;\subset\; T_q \mathcal S^N_{\alpha(q)} \times \{0\}.
\]
On the other hand, since $T_{x_i}G_N\supset \ker Df_{q_i}\times\{0\}$ for all $i$ and
$\ker Df_{q_i}\to T$ in the Grassmannian, we also have $\tau\supset T\times\{0\}$.
Therefore
\[
T\times\{0\}\;\subset\;\tau\;\supset\; T_q \mathcal S^N_{\alpha(q)}\times\{0\}.
\]
Because $\tau$ is a perplex hyperplane in $\PP^{n+1}$, whereas both $T\times\{0\}$ and
$T_q \mathcal S^N_{\alpha(q)}\times\{0\}$ are subspaces of $\PP^n\times\{0\}$, the assumption (c)
(that $T$ does not contain $T_q \mathcal S^N_{\alpha(q)}$) forces
\begin{equation}\label{eq:tau-vertical}
\tau \;=\; \PP^n \times \{0\}.
\end{equation}

We now derive a contradiction from \eqref{eq:tau-vertical} using Theorem~\ref{theo_Loja} and Lemma~\ref{lemma_limites}.
By Theorem~\ref{theo_Loja} there exist $0<\theta<1$, $C>0$, and (after shrinking $U$ if necessary) a neighborhood
$U\ni 0$ such that
\[
\|\nabla f(p)\|_m \;\ge\; C\,\|f(p)\|_m^\theta
\qquad\text{for all }p\in U.
\]
Evaluating at $p=q_i$ and using $f(q_i)=t_i^N$ gives
\[
\|\nabla f(q_i)\|_m \;\ge\; C\,\|t_i^N\|_m^\theta.
\]
Hence, for each $i$ with $t_i\neq 0$,
\begin{equation}\label{eq:ratio}
\frac{\|\nabla f(q_i)\|_m}{N\,\|t_i^{N-1}\|_m}
\;\ge\;
\frac{C}{N}\,\frac{\|t_i^N\|_m^\theta}{\|t_i^{N-1}\|_m}.
\end{equation}
By Lemma~\ref{lemma_norm_limits}, choosing $N>\tfrac{1}{1-\theta}$ ensures that the right-hand side
of \eqref{eq:ratio} tends to $+\infty$ as $i\to\infty$ (recall $t_i\to 0$ since $x_i\to (q,0)$).
Therefore
\[
\frac{\|\nabla f(q_i)\|_m}{N\,\|t_i^{N-1}\|_m}\;\longrightarrow\;+\infty.
\]

Consider now the gradient of $g_N$ at $x_i$ written in $\PP^{\,n+1}$–coordinates:
\[
\nabla g_N(x_i) \;=\; \big(\nabla f(q_i),\;-\;N\,t_i^{N-1}\big).
\]
Decompose $\nabla g_N(x_i)=(\alpha_i,\beta_i)$ with $\alpha_i:=\nabla f(q_i)\in\PP^n$ and
$\beta_i:=-N\,t_i^{N-1}\in\PP$. The previous limit implies
\[
\frac{\|\alpha_i\|_m}{\|\beta_i\|_m}\;\longrightarrow\;+\infty.
\]
Since each $x_i$ is a regular point of $g_N$ (as observed above), Corollary~\ref{cor_ts} applies to $g_N$,
and we may invoke Lemma~\ref{lemma_limites} to conclude that the hyperplanes
$\ker D(g_N)_{x_i}=T_{x_i}G_N$ \emph{do not} converge to $\PP^n\times\{0\}$.
This contradicts \eqref{eq:tau-vertical}.
\end{proof}

\subsection{Milnor--Lê fibration}

For $\varepsilon,\eta>0$, let $\B_\varepsilon$ be the closed ball in $\R^{2n}$ of radius $\varepsilon$ and $\D_\eta$ the closed disk in $\R^2$ of radius $\eta$.  
Given a real-analytic map $f:\R^{2n}\to\R^2$, let $\Delta_f$ denote the discriminant, i.e. the image of its critical set.

As in the complex case, Theorem~\ref{theo_stratification} ensures that for small $\varepsilon>0$, the restriction of $f$ to $\mathbb S_\varepsilon\setminus f^{-1}(\Delta_f)$ is a submersion.  
By Ehresmann’s fibration theorem for manifolds with boundary (see \cite[Thm.~8.2]{Hirsch1976}), we obtain the following analogue of Milnor’s classical theorem (see \cite{Seade} for a survey).

\begin{theo}\label{theo_MilnorFibration}
Let $\PP$ be a nondegenerate perplex algebra, and let $f : \PP^n \to \PP$ be a $\PP$–analytic function with $f(0)=0$.  
There exist real numbers $0<\eta<\varepsilon$ such that the restriction
\[
f:\ f^{-1}(\D_\eta\setminus\Delta_f)\cap \B_\varepsilon
\;\longrightarrow\;\D_\eta\setminus\Delta_f
\]
is the projection of a smooth locally trivial fibration.
\end{theo}

\section{Perspectives and open problems in perplex geometry and singularities}\label{sec:open}

The results obtained in this work open the way to a broader research program. 
They suggest that perplex analysis provides a genuine intermediate framework between real and complex theories, 
capable of supporting new tools in singularity theory and beyond. 
We conclude by formulating a few natural questions that highlight the main challenges and possible directions for future developments.

\medskip
\noindent\textbf{(Q1) Topology of the Milnor fiber(s).}
The fibration in Theorem~\ref{theo_MilnorFibration} produces, for each connected component 
$C$ of $\D_\eta \setminus \Delta_f$, a Milnor–Lê fiber
\[
F_{C,\varepsilon}\;:=\;f^{-1}(c)\cap \B_\varepsilon,
\qquad c\in C,
\]
well defined up to diffeomorphism for $0<\eta\ll\varepsilon\ll1$.  
Unlike the complex case, the discriminant $\Delta_f$ may disconnect the disk $\D_\eta$, 
so several non-equivalent fibers can coexist.  
What topological invariants (homotopy type, Betti numbers, or analogues of polar multiplicities) 
can be associated to each fiber, and how do they depend on algebraic data of $f$ and of the chosen 
perplex algebra $\PP$?  
In particular, when $\PP$ is not a field, do zero divisors introduce new topological phenomena 
absent in the complex case?

\medskip
\noindent\textbf{(Q2) Topological invariants for perplex functions.}
In the classical setting, the Milnor number $\mu$ measures the complexity of an isolated singularity
via the dimension of a Jacobian algebra \cite{Mi, Seade}, while the Lê numbers capture the topology of the Milnor fiber of non-isolated singularities \cite{Massey1995}.  
In the perplex setting, the presence of zero divisors prevents the occurrence of genuinely isolated singularities unless $\PP\cong\C$.  
Thus, the natural question is: Can one develop \emph{perplex Lê numbers} that capture the topology of the Milnor fibers, and that specialize to the classical Lê numbers in the complex case while detecting new phenomena arising from zero divisors?

\medskip
\noindent\textbf{(Q3) Classification of singularities.}
Arnol’d’s celebrated program classifies complex singularities up to right-left equivalence, 
with simple and unimodal series (see e.g.~\cite{Arnold1972,ArnoldGuseinZadeVarchenko1985,ArnoldGuseinZadeVarchenko1988}).  
Is there a meaningful analogue in the perplex setting, where the presence of zero divisors 
and the variability of $(a,b)\in\mathcal P$ may lead to new deformation patterns 
and moduli of singularities?

\medskip
\noindent\textbf{(Q4) Perplex stratifications.}
Whitney and Thom conditions play a central role in the stability of singularities and 
in the existence of Milnor fibrations (see for instance Whitney’s foundational work~\cite{Whitney1965}, 
Thom’s conditions~\cite{Thom1969}, and Mather’s theory of stability~\cite{Mather1970}).  
We proved here that $\PP$--analytic functions admit good stratifications in the sense of Hamm–Lê~\cite{HL}.  
Can one develop a systematic theory of \emph{perplex stratifications}, 
possibly refining the classical conditions to capture the special algebraic features of $\PP$?

\medskip
\noindent\textbf{(Q5) Geometry of perplex analytic sets.}
Complex analytic sets admit a rich geometric theory, including dimension theory, irreducibility, 
normalization, and desingularization (see e.g.~Hironaka’s resolution theorem~\cite{Hironaka1964} and 
Chirka’s monograph~\cite{Chirka1989}).  
What are the corresponding notions for sets defined by $\PP$--analytic equations?  
Do these sets admit meaningful decompositions or geometric invariants that generalize the complex case?

\medskip
\noindent\textbf{(Q6) Perplex dynamics.}
Iterated maps in the complex plane give rise to Julia and Fatou sets, central in holomorphic dynamics 
(see Julia’s pioneering work~\cite{Julia1918}, Fatou’s classical studies~\cite{Fatou1919}, and modern expositions 
by Milnor~\cite{Milnor2006} and Carleson--Gamelin~\cite{Carleson1993}).  
What dynamical phenomena emerge from iterating perplex analytic maps? Two new directions appear in the perplex setting: 
the variation of the algebra $(a,b)\in\mathcal P$, which gives rise to families of dynamical systems
parametrized by $\mathcal P$, and the higher-dimensional case $\PP^n$, 
where the presence of zero divisors may lead to new orbit structures or bifurcation phenomena.  
What kinds of limit sets, stability notions, and parameter-space bifurcations arise in these broader contexts?

\medskip
\noindent\textbf{(Q7) Global geometry and moduli.}
Families of complex singularities often organize into moduli spaces, governed by deformation theory 
(see for instance the work of Kuranishi~\cite{Kuranishi1962}, Kodaira~\cite{Kodaira1963}, 
and Pinkham~\cite{Pinkham1974}).  
Is it possible to construct moduli of $\PP$--analytic singularities, incorporating the additional freedom of varying $(a,b)\in\mathcal P$?  
Could such spaces provide a new bridge between real and complex singularity theories, 
with new deformation patterns reflecting the intermediate nature of perplex geometry?

\medskip
\noindent Altogether, these questions suggest that the theory of perplex analytic functions 
opens a genuinely new field at the intersection of real and complex geometry.  
Exploring their fibers, invariants, and classifications may reveal topological phenomena 
that interpolate between the two classical worlds, and perhaps expose structures 
that are invisible from either side alone.  
In this sense, perplex analysis not only extends Milnor’s vision, but also points 
toward a broader singularity theory still to be developed.

\end{document}